\documentclass[10pt,leqno]{amsart}
\usepackage{graphicx}
\usepackage{indentfirst,csquotes}

\usepackage{amssymb,amsthm,amsmath}
\usepackage{xcolor,paralist,hyperref,titlesec,fancyhdr,etoolbox}
\newtheorem{theorem}{Theorem}[section]
\newtheorem{definition}[theorem]{Definition}

\newtheorem{lemma}[theorem]{Lemma}
\newtheorem{proposition}[theorem]{Proposition}

\titleformat{\section}[hang]
 {\normalfont\scshape\centering}
 {\thesection. }{0pt}{}
 \titlespacing{\section}
 {\parindent}{*2}{\wordsep}

\hypersetup{ colorlinks=true, linkcolor=black, filecolor=black, urlcolor=black }

\begin{document}
\title{Sums of Hurwitz Class Numbers and newform of weight 2 and level 49} 
\author[Initial Surname]{FangMin Guo}
\date{\today}
\email{161240020@smail.nju.edu.cn}
\maketitle

\let\thefootnote\relax

\begin{abstract}
We consider sums of Hurwitz class number $H_{m,M}(n)=\sum_{t\equiv m (\text{mod} M)}{H(4n-t^2)}$, where $H(N)$ denotes the Hurwitz class number. In this article, we consider the case of $M=7$. By completing the mixed mock modular form generated by $H_{m,7}(n)$, We obtain the formula of modular forms consist of a computable part and a part from newform 49.2.a.a whose prime terms of Fourier expansion has a connection with with $p=x^2+7y^2$ $(p\equiv 1,2,4 \mod 7)$.
\end{abstract} 

\bigskip

\section{INTRODUCTION}

The Hurwitz class number, first induced by Kronecker, Hurwitz, Gierster(\cite{history}, Vol.III, Chapter.VII), can be defined as weighted sum of class numbers of orders(\cite{primenxy}, Chapter.14):\[H(\mathcal{O})=\sum_{\mathcal{O'\subset\mathcal{O}\subset\mathcal{O}_K}}\frac{2}{\#(\mathcal{O'^*})}h(\mathcal{O'}),\]where $\mathcal{O}$ is an order in the integral order $\mathcal{O}_K$. The orders with the same discriminant shares the same $H(\mathcal{O})$, therefore $H(D)$, where $D$ denotes the discriminant, is well-defined. This definition is equivalent with the weighted sum:\[H(D)=\sum_{Q\in\Omega/\rm{SL}_2(\mathbb{Z})}\frac{1}{w_Q},\quad D\equiv 0,-1\mod 4,\]where $\Omega$ represents the set of all integral binary quadratic form of discriminant $D$, $w_Q$ equals $1/2$ if $Q=a(x^2+y^2)$, $1/3$ if $Q=a(x^2+xy+y^2)$ and $1$ in other case. Moreover, we define $H(0)=-1/12$ and $H(D)=0$ if $D\equiv 1,2\mod 4$ for convenience. 

There are several equations about sums of Hurwitz class numbers. The well-known Hurwitz-Kronecker relation(\cite{primenxy}, Chapter.14), first found in 1860, can be written as $\sum_{a\in\mathbb{Z}}H(4n-a^2)=2\sum_{d|n}d-\sum_{d|n}\min(d,\frac{n}{d})$, or in prime case, \[\sum_{a\in\mathbb{Z}}H(4p-a^2)=2p.\] One can prove this equation by computing the degree of modular equation $\Phi_m(X,X)$(\cite{primenxy}, Chapter.13). Similar phenomenon occurs when considering sums that taking modular. Define \[H_{m,M}(n):=\sum_{\substack{m\in\mathbb{Z}\\[1pt]a\equiv m\mod M}}H(4n-a^2).\] B. Brown et al. work considers the case m=2,3,4,5,7\cite{Brown2008EllipticCM}. For example, \[H_{m,2}(p)=\begin{cases}
\frac{4p-2}{3} & \text{if } m = 0 \\
\frac{2p+2}{3} & \text{if } m = 1.
\end{cases}\]
They give proof of all cases when $m=2,3,4$ and several cases when $m=5,7$. They conjectured all other case of $m=5$ and several cases of $m=7$. The pattern of other cases when $m=7$ was unknown in that article. In 2013 and 2014, K. Bringmann AND B. Kane\cite{BRINGMANN_2018} and Mertens\cite{Mertens_2016} proved the conjectured case about $m=5,7$ using knowledge about mixed mock modular form.

Using similar method of mixed mock modular form, In 2024, M. Zindulka\cite{zindulka2024sumshurwitzclassnumbers} prove that when $m=6,8$, a new pattern of equation occurs. The expression of $H_{m,M}(p)$ includes a term linearly depend on $p$ as usual and an "error term", expresses as $\chi(x)x$, where $\chi$ denote a non-principal real character and $x$ is the unique positive number satisfies $p=x^2+ly^2$, where $l$ is an integer depends on $m,M$. For example, \[H_{0,6}(p)=\begin{cases}
\frac{p+1}{3}+\frac{1}{3}\chi_{-3}(x)x & \text{if } p\equiv 1\mod 3 \\
\frac{2p-4}{3} & \text{if } p\equiv 2\mod 3.
\end{cases}\]

In this article, we will give all formula of $H_{m,7}(p)$, using the method of mixed mock modular form.
Here is the main result of this article:
\begin{theorem}\label{main}
    Let $p>2$ be a prime and $p\equiv 1,2,4 \mod 7$. Then there exist a unique positive integer $x$ such that $p=x^2+7y^2$. Let $\chi_{-7} $ denotes the non principle real character modulo 7. We have\[
    H_{a,7}(p)=\begin{cases}
\frac{p+1}{4}+\frac{1}{2}\chi_{-7}(x)x & \text{if } a=0 \\
\frac{p+1}{4}-\frac{1}{2}\chi_{-7}(x)x & \text{if } a=\pm 1, p\equiv 4\mod 7, \text{or }a=\pm 2, p\equiv 1\mod 7, \\ & \text{or }a=\pm 3, p\equiv 1\mod 7,\\
\frac{7p+7}{24}+\frac{1}{4}\chi_{-7}(x)x & \text{if } a=\pm 1, p\equiv 2\mod 7, \text{or }a=\pm 3, p\equiv 4\mod\\
\frac{7p-17}{24}+\frac{1}{4}\chi_{-7}(x)x & \text{if } a=\pm 2, p\equiv 1\mod 7.
\end{cases}
    \]
\end{theorem}
The pattern of this result is similar with case $m=6,8$. When $m=6,8$, The error terms are from CM cusp form $\Psi_3(\chi_{-3},\tau)\in\Gamma_0(36)$ and $\Psi_4(\chi_{-4},\tau)\in\Gamma_0(64)$ respectively, where $\Psi_k(\chi,\tau)=\frac{1}{2}\sum_{n=1}^{\infty}(\sum_{x^2+ky^2=n}\chi(x)x) q^n$. In the case $m=7$, the error terms is from CM cusp form with label $49.2.a$\cite{lmfdb}. This modular form is different from $\Psi_7(\chi_{-7},\tau)\in\Gamma_0(4\cdot49)$, but has a connection with it.

This paper is organized as follows. In section 2, we briefly introduce the basic concepts of mixed mock modular form. In section 3, we obtain the equation of the modular form generated by $H_{m,7}(n)$. In the last section, we compute the prime term of cusp form $49.2.a.a$ to deduce Thm\ref{main}.

\section{PRELIMINARIES}

In this chapter, we introduce basic knowledge of mixed mock modular forms and several operators of modular forms we will use latter. After this, we state several theorems which serve the modular form generated by $H_{m,7}(n)$. We shall start from Hurwitz class number and its generate $q$-series. This $q$-series is not a modular form but a mock modular form. We introduce the definition of mock modular form first.

\begin{definition}
Let $\mathbb{H}$ be the upper half plane, $\tau=x+iy\in\mathbb{H}$, $x,y\in\mathbb{R}$. Fix $k\in\frac{1}{2}\mathbb{Z}$ and a smooth function $f:\mathbb{H}\xrightarrow{}\mathbb{C}$. Let $\gamma=\left(\begin{matrix}a & b\\c & d\end{matrix}\right)\in \text{SL}_2(\mathbb{Z})$. If $k\in\frac{1}{2}+\mathbb{Z}$ we additionally need $\gamma\in\Gamma_0(4)$. The weight $k$ slash operator $|_k$ is defined by\[(f|_k\gamma)(\tau)=\begin{cases}
    (c\tau+d)^{-k}f\left(\frac{a\tau+b}{c\tau+d}\right) & \tau\in\mathbb{Z}\\
    (c\tau+d)^{-k}\left(\frac{c}{d}\right)\epsilon_df\left(\frac{a\tau+b}{c\tau+d}\right) & \tau\in\frac{1}{2}+\mathbb{Z},
\end{cases}\]where $\left(\frac{c}{d}\right)$ is the extended Legendre symbol, and $\epsilon_d=\begin{cases}
    1 & d\equiv 1\mod 4 \\
    i & d\equiv 3\mod 4
\end{cases}$.
\end{definition}

\begin{definition}{\rm(pure mock modular form)}\cite{dabholkar2014quantumblackholeswall}\\
    Let $\mathbb{H}$ be the upper half plane, $k\in\frac{1}{2}\mathbb{Z}$, $\Gamma$ be a congruent subgroup of $\rm{SL}_2(\mathbb{Z})$. We define the (weakly holomorphic) pure mock modular form of weight $k\in\frac{1}{2}\mathbb{Z}$ to be the first term of the pair $(h,g)$, where\\
    \rm (1) \it $h:\mathbb{H}\xrightarrow{}\mathbb{C}$ is a holomorphic function and grows at most exponential speed at all the cusps of $\Gamma$.\\
    \rm (2) \it the function $g$, called the shadow of $h$, is a holomorphic modular form of weight $2-k$.\\
    \rm (3) \it the sum $\hat{h}=h+g^*$, called the completion of $h$, transforms like a modular form of weight $k$ on $\Gamma$, which means $\hat{h}|_k\gamma=\hat{h}$ for all $\gamma\in \Gamma$.(cite  zagier quantum blackhole *********) The $g^*(\tau)$ above, called the non-holomorphic Eichler integral, is the solution of the differential function\[(4\pi\text{Im}(\tau))^k\frac{\partial g^*(\tau)}{\partial\bar{\tau}}=-2\pi i\overline{g(\tau)}.\]Moreover, if the Fourier expansion of $g$ is $\sum_{n\geq 0}b_nq^n$, we can choose $g^*$ to be \[g^*(\tau)=\bar{b_0}\frac{(4\pi \text{Im}(\tau))^{-k+1}}{k-1}+\sum_{n>0}n^{k-1}\bar{b_n}\Gamma(1-k,4\pi n\text{Im}(\tau))q^{-n},\]where $\Gamma(1-k,x)=\int_{x}^{\infty}t^{-k}e^{-t}dt$ is the incomplete Gamma function. If $k=1$, we shall replace the first term of $g^*(\tau)$ with $-\bar{b_0}\log(4\pi\text{Im}(\tau))$. In the case that $k>1$, to avoid divergent, we shall replace $g^*(\tau)$ with rountine integral \[\left(\frac{i}{2\pi}\right)^{k-1}\int_{-\bar{\tau}}^{\infty}(z+\tau)^{-k}\overline{g(-\bar{z})}dz.\]
\end{definition}

One can also define mock modular form by harmonic maass form. To put it simply, a harmonic maass form may canonically decompose into a holomorphic part and a non-holomorphic part. A mock modular form is the holomorphic part of a harmonic maass form. For more details of this traditional definition, see \cite{maass}\cite{bruinier2003geometricthetalifts}.

Back to Hurwitz class numbers. For $\tau$ in the upper half plane, let $q=e^{2\pi i\tau}$. We define the generating series of Hurwitz class numbers: \[\mathcal{H}(q):=\sum_{n=0}^{\infty}H(n)q^n.\]$\mathcal{H}(\tau)$ is not a modular form, instead it is a mock modular form. By Zagier's work\cite{Hirzebruch1976IntersectionNO} in 1976, if we write \[\mathcal{F}(\tau):=\mathcal{H}(\tau)+(\text{Im}(\tau))^{-\frac{1}{2}}\sum_{f=-\infty}^{\infty}\beta(4\pi f^2\text{Im}(\tau))e^{-2\pi if^2}\], where $\beta=\frac{1}{16\pi}\int_{1}^{\infty}u^{-\frac{3}{2}}e^{-xu}du$, then $\mathcal{F}$ satisfies weak modularity of weight $\frac{3}{2}$ on $\Gamma_0(4)$. Here $\mathcal{F}$ is the complete of $\mathcal{H}$, and the classical theta function $\theta(\tau)=\sum_{n\in\mathbb{Z}}q^{n^2}$ is the shadow of $\mathcal{H}$.

As an extension of pure mock modular form, the conception of mixed mock modular form is introduced. 
\begin{definition}{\rm(mixed mock modular form)}\cite{dabholkar2014quantumblackholeswall}\\
    A mixed mock modular form of mixed weight $k|l$ on some congruent subgroup $\Gamma\subset\rm{SL}_2(\mathbb{Z})$ is a holomorphic function $h:\mathbb{H}\xrightarrow{}\mathbb{C}$, with polynomial growth at all the cusps, can be completed with $\hat{h}=h+\sum_{j}f_j g_j^*$, where $\hat{h}$ transforms like modular forms of weight $k$, $f_j$ is a modular form of weight $l$ and $g_j$ ia a modular form of weight $2-k+l$.
\end{definition}

Its easy to find that if $h$ is a mock modular form of weight $k$ on $\Gamma$ with polynomial growth st the cusps with completion $\bar{h}=h+g^*$, and $f$ is a modular form of weight $l$ on $\Gamma$, then by $\hat{h}f=hf+g^*f$, $hf$ is a mixed mock modular form of weight $k+l|l$.

Before proceeding the introduction, we introduce several operators of modular form.
\begin{definition}
Let $f$ be a modular form of weight $k\in\frac{1}{2}\mathbb{Z}$ on some congruence subgroup $\Gamma$ of $\text{SL}_2(\mathbb{Z})$. Suppose $f$ has Fourier expansion $f=\sum_{n\geq0}a(n)q^n$.\\
\rm (1) \it For a positive integer $M$ The $U_M$ operator is defined as $(f|U_M)(\tau):=\sum_{n\geq0}a(Mn)q^n$.\\
\rm (2) \it For a positive integer $M$ The $V_M$ operator is defined as $(f|V_M)(\tau):=\sum_{n\geq0}a(Mn)q^{Mn}$.\\
\rm (3) \it For a positive integer $M$ and a integer $m$ The Sieving operator $S_{M,m}$ is defined as\\ $(f|S_{M,m})(\tau):=\sum_{n\equiv m\mod M}a(n)q^{n}$.\\
\rm (4) \it For a character $\chi$ Th operator $\otimes$ is defined as $(f\otimes \chi)(\tau):=\sum_{n\geq 0}a(n)\chi(n)q^{n}$.
\end{definition}
\begin{proposition} Let $\Gamma_1$ and $\Gamma_2$ be congruence subgroups of $\text{SL}_2(\mathbb{Z})$. Let $k\in\frac{1}{2}\mathbb{Z}$.\\ If $k\in\frac{1}{2}+\mathbb{Z}$, we need $\Gamma_1\subset\Gamma_0(4)$ additionally.\\
    \rm (1) \it Let integer $N_2|N_1$, then $U_M$ maps modular forms of weight $k$ on $\Gamma_1=\Gamma_0(N_1)\cap\Gamma_1(N_2)$ to $\Gamma_2=\Gamma_0(\text{lcm}(N_1,M))\cap\Gamma_1(N_2)$.\\
    \rm (2) \it Let integer $N_2|N_1$, then $V_M$ maps modular forms of weight $k$ on $\Gamma_1=\Gamma_0(N_1)\cap\Gamma_1(N_2)$ to $\Gamma_2=\Gamma_0(\text{lcm}(N_1,M^2,MN_2))\cap\Gamma_1(N_2)$.\\
    \rm (3) \it Let integer $N_2|N_1$, then $S_{m,M}$ maps modular forms of weight $k$ on $\Gamma_1=\Gamma_0(N_1)\cap\Gamma_1(N_2)$ to $\Gamma_2=\Gamma_0(\text{lcm}(N_1,M^2,MN_2))\cap\Gamma_1(\text{lcm}(N_2,M))$.\\
    \rm (4) \it $\cdot\otimes\chi$ maps modular forms of weight $k$ and character $\phi$ on $\Gamma_1=\Gamma_0(N)$ to modular forms of weight $k$ and character $\phi\chi^2$ on $\Gamma_2=\Gamma_0(\text{lcm}(N,M^2)$.\\
\end{proposition}
\begin{proof}
    we represent above operators with the double coset operators, which are operator of modular forms of weight $k$ between given congruence subgroups of $\text{SL}_2(\mathbb{Z})$. See \cite{Diamond2008AFC} 5.1.3 for more detail of this operator.\\
\rm (1) \it We write down the Fourier series of $f|U_M$:\begin{align*}
f|{U_M}(\tau)&=\sum_{n\geq 0}a(Mn)q^n \\
&=\sum_{n\geq 0}a(n)(\frac{1}{M}\sum_{l=0}^{M-1}e^{\frac{2\pi inl}{M}})e^{\frac{2\pi in\tau}{M}} \\
&=\frac{1}{M}\sum_{l=0}^{M-1}f(\frac{\tau+l}{M})\\
&=\frac{1}{M}\sum_{l=0}^{M-1}f|_k\left(\begin{matrix}\frac{1}{M} & \frac{l}{M}\\0 & 1\end{matrix}\right).
\end{align*}\\Set $\alpha=\left(\begin{matrix}\frac{1}{M} & 0\\0 & 1\end{matrix}\right)$, $\beta_j=\left(\begin{matrix}\frac{1}{M} & \frac{j}{M}\\0 & 1\end{matrix}\right)$. We prove that $\cup_j\Gamma_1\beta_j=\Gamma_1\alpha\Gamma_2$ and for $i\neq j$ $\Gamma_1\beta_i\cap\Gamma_1\beta_j=\emptyset$.\\
Because $\alpha\left(\begin{matrix}1 & j\\0 & 1\end{matrix}\right)=\left(\begin{matrix}\frac{1}{M} & \frac{j}{M}\\0 & 1\end{matrix}\right)=\beta_j$, $\Gamma_1\alpha\Gamma_2\supset\cup_j\Gamma_1\beta_j$. To show $\Gamma_1\alpha\Gamma_2\subset\cup_j\Gamma_1\beta_j$, suggest $\left(\begin{matrix}a & b\\c & d\end{matrix}\right)\in\Gamma_2=\Gamma_0(\text{lcm}(N_1,M))\cap\Gamma_1(N_2)$. Suppose there is a matrix $\left(\begin{matrix}a' & b'\\c' & d'\end{matrix}\right)$ satisfies \[\left(\begin{matrix}\frac{1}{M} & \frac{l}{M}\\0 & 1\end{matrix}\right)\left(\begin{matrix}a & b\\c & d\end{matrix}\right)=\left(\begin{matrix}a' & b'\\c' & d'\end{matrix}\right)\left(\begin{matrix}\frac{1}{M} & \frac{l'}{M}\\0 & 1\end{matrix}\right),\] we find out that \[\begin{cases}
    c'=Mc,\\
    a'=a+lc,\\
    d'=d-l'c,\\
    b'=\frac{b+ld-l'(a+lc)}{M}.
\end{cases}\]
Using $ad-bc=1$ and $c\equiv 0\mod M$ we conclude that $ad\equiv1\mod M$. Therefore take $l'=bd+ld^2$, $b'$ becomes an integer. Further check indicate that $\left(\begin{matrix}a' & b'\\c' & d'\end{matrix}\right)\in\Gamma_1$, therefore $\Gamma_1\alpha\Gamma_2\subset\cup_j\Gamma_1\beta_j$.\\
Suppose $\left(\begin{matrix}a & b\\c & d\end{matrix}\right),\left(\begin{matrix}a' & b'\\c' & d'\end{matrix}\right)\in\Gamma_1$ and $\left(\begin{matrix}a & b\\c & d\end{matrix}\right)\left(\begin{matrix}\frac{1}{M} & \frac{l}{M}\\0 & 1\end{matrix}\right)=\left(\begin{matrix}a' & b'\\c' & d'\end{matrix}\right)\left(\begin{matrix}\frac{1}{M} & \frac{l'}{M}\\0 & 1\end{matrix}\right)$. We get $a=a'$, $c=c'$, and $\begin{cases}
    \frac{a}{M}(l-l')=b'-b,\\
    \frac{c}{M}(l-l')=d'-d,.\end{cases}$. $l-l'$ must be 0, otherwise $\frac{a-c}{M}(l-l')\in\mathbb{Z}$, contradict with $(a,c)=1$.\\
    Therefore $MU_M=\Gamma_1\alpha\Gamma_2$ is a double coset operator sending $\mathcal{M}_k(\Gamma_1)$ to $\mathcal{M}_k(\Gamma_2)$.
\\
\rm (2) \it $V_M$=$S_{0,M}$, and the proof of \rm (2) \it is deduced from below.\\
\rm (3) \it 
See \cite{zindulka2024sumshurwitzclassnumbers} 3.5.\\
\rm (4) \it Decompose $\cdot\otimes\chi$ into sum of sieving operators.
\end{proof}

The key tool we use in the proof is introduced by Rankin-Cohen bracket, which is defined as follow:
\begin{definition}
    For modular form $f$ transforms like modular form of weight $k\in\frac{1}{2}\mathbb{Z}$, $g$ transforms like modular form of weight $l\in\frac{1}{2}\mathbb{Z}$, the Rankin-Cohen bracket of $(f,g)$ is defined by \[[f,g]_n=[f,g]^{(k,l)}_n=\sum_{r+s=n}(-1)^s\tbinom{k+n-1}{r}\tbinom{l+n-1}{s}f^{(s)}(\tau)g^{(r)}(\tau),\] where $f^{(s)}(\tau)=(\frac{1}{2\pi i}\frac{d}{d\tau})^sf$. Then $f$ transforms like modular form of weight $k+l+2n$\cite{Cohen1975SumsIT}.
\end{definition}

Beside the Rankin-Cohen bracket of two modular forms, we may also consider the Rankin-Cohen bracket of one mock modular form and a modular form. Using this operator, we are able to express the $q$-series generated by $H_{m,M}(n)$. Define \[\mathcal{H}_{m,M}(\tau):=\sum_{n\geq 0}H_{m,M}(n)q^n\] and let $\theta_{m,M}(\tau):=\sum_{\substack{n\in\mathbb{Z}\\[1pt]n\equiv m\mod M}}q^{n^2}$ to be the "modified" theta series. Then, we can represent $\mathcal{H}_{m,M}$ using $\mathcal{H}$ and $\theta_{m,M}$ and the Rankin-Cohen bracket of level 0:\[
\mathcal{H}_{m,M}(\tau)=(\mathcal{H}\theta_{m,M})|U_4.
\]This formula indicates $\mathcal{H}_{m,M}(\tau)$ is a mixed mock modular form. To compute the coefficients of $\mathcal{H}_{m,M}(\tau)$ effectively, its worthwhile to complete this mixed mock modular form. We refer to  B. Kane AND S. Pujahari's work\cite{kane2022distributionmomentshurwitzclass}:
\begin{theorem}\label{27}
    Set (The sum symbol with * means the term $s=0$ takes weight $\frac{1}{2})$\[\lambda_{l,m,M}(n)=\sum_{\pm}\sum_{{\substack{t>s\geq 0\\[1pt]t^2-s^2=n\\[1pt]t\equiv \pm m\mod M}}}^{\quad\quad*}(t-s)^l\] and \[\Lambda_{l,m,M}(\tau)=\sum_{n\geq q}\lambda_{l,m,M}(n)q^n.\]For $k\in\mathbb{Z}_{\geq 0}$, $m\in\mathbb{Z}$ and $M\in\mathbb{Z}_{> 0}$, the function \begin{equation}\label{h0}
        ([\mathcal{H},\theta_{m,M}]_k+2^{-1-2k}\binom{2k}{k}\Lambda_{2k+1,m,M})|U_4
    \end{equation}
    is a holomorphic cusp form of weight $2+2k$ on $\Gamma_0(4M^2)\cap\Gamma_1(M)$ if $M\nmid m$ and $\Gamma_0(4M^2)$ if $M\mid m$ if $k>0$; if $k=0$, the function is quasimodular on that congruence subgroup.
\end{theorem}

Here the concept of quasimodular form is from \text{almost holomorphic modular form}. If a function $F:\mathbb{H}\xrightarrow{}\mathbb{C}$ satisfies weak modularity of weight $k\in\frac{1}{2}\mathbb{Z}$ on some congruent subgroup $\Gamma$ of $\text{SL}_2(\mathbb{Z})$ and there exists holomorphic functions $F_j$ $(0\leq j\leq l)$ such that $F=\sum_{j=0}^{l}\frac{F_j(\tau)}{Im(\tau)^j}$. The term $F_0$ is called a \text{quasimodular form}. For more details of quasimodular form, one may check \cite{quasiintro}.

Note that in the case $k=0$ this theorem shows that \begin{equation}\label{h1}
    (\mathcal{H}_{m,7}+\frac{1}{2}\Lambda_{1,m,7})|U_4
\end{equation} is a quasimodular form of weight $2$ on $\Gamma_0(4\cdot 49)\cap\Gamma_1(7)$ if $m\not\equiv 0 \ (\text{mod} 7)$ or $\Gamma_0(4\cdot 49)$ if $m\equiv0 \ (\text{mod} 7)$. If furthermore we act the sieving operator $S_{a,7}$ $(a\not\equiv 0\mod 7)$ on the quasimodular form, because $S_{m,7}$ vanish the $n=0$ term, \[(\mathcal{H}_{m,7}+\frac{1}{2}\Lambda_{1,m,7})|U_4|S_{a,7} \quad \text{and} \quad (\mathcal{H}_{m,7}+\frac{1}{2}\Lambda_{1,m,7})|U_4\otimes\chi_{-7}^2
\] are modular forms of weight 2 on $\Gamma_0(4\cdot 49)\cap\Gamma_1(7)$ if $m\not\equiv 0 \ (\text{mod} 7)$ or $\Gamma_0(4\cdot 49)$ if $m\equiv0 \ (\text{mod} 7)$, where $\chi_{-7} $ is the non-principal real character with period 7.


In the next chapter, we start calculate the two sides of (\ref{h0}) in Thm\ref{27}. By sieving, we transform (\ref{h1}) into a modular form, which may be represented as sum of several modular forms in the modular space that easy to compute.


\section{TERMS OF FOURIER EXPANSIONS}

In the previous chapter, we obtain a modular form of weight 2 which coefficient of terms of Fourier series is related to $H_{m,M}(n)$ $(n\not\equiv 0\mod 7)$ by sieving. In this chapter, first we compute $\frac{1}{2}\Lambda_{1,m,7}|U_4|S_{a,7}$ $(a\not\equiv 0\mod 7)$, then we are going to write our modular forms as linear combination of series that are easy to compute.

Using the following proposition, one can rewrite $\frac{1}{2}\Lambda_{1,m,7}|U_4$ into a more direct form.

\begin{proposition}\label{31}\cite{mertens2013mockmodularformsclass}
    Let\[D_l^{(p,a)}(\tau):=\sum_{n=1}^{\infty}\phi_l^{(p,a)}(n)q^n,\]
    where\[\phi_l^{(p,a)}(n):=\sum_{\substack{d\mid n\\[1pt]d\leq\sqrt{n}\\[1pt]d\equiv -a\mod p}}d^l+\sum_{\substack{d\mid n\\[1pt]d<\sqrt{n}\\[1pt]d\equiv a\mod p}}d^l.\]Then for $k\in\mathbb{Z}_{\geq 0}$ and $m\neq 0$,\begin{align*}
            \Lambda_{2k+1,m,p}\mid U(4)(\tau)= & 2^{2k+1}[\sum_{a\not\equiv \pm m\mod p}(\frac{1}{2}(D_{2k+1}^{(p,\frac{m-a}{2})}+D_{2k+1}^{(p,\frac{a-m}{2})})\mid S_{p,\frac{m^2-a^2}{4}})(\tau)+\\ & ((D_{2k+1}^{(p,m)}+D_{2k+1}^{(p,-m)})\mid S_{p,0})(\tau)+p^{2k+1}(D_{2k+1}^{(1,0)}\mid V(p))(\tau)],
    \end{align*}
    and for $m=0$, \[
    \Lambda_{2k+1,0,p}\mid U(4)(\tau) =2^{2k+1}[\sum_{a\not\equiv 0\mod p}(D_{2k+1}^{(p,\frac{a}{2})}\mid S_{p,-\frac{a^2}{4}})(\tau) +p^{2k+1}(D_{2k+1}^{(1,0)}\mid V(p^2))(\tau)].
    \]
\end{proposition}

The coefficients of the right side of $\frac{1}{2}\Lambda_{1,m,7}|U_4\otimes\chi_{-7}^2$ now becomes easily computable, especially the prime terms' coefficients. 

Now we are going to consider $(\mathcal{H}_{m,7}+\frac{1}{2}\Lambda_{1,m,7})|U_4|S_{a,7}$ or $(\mathcal{H}_{m,7}+\frac{1}{2}\Lambda_{1,m,7})|U_4\otimes\chi_{-7}^2$ in the finite dimensional linear space $\mathcal{M}_2(\Gamma)$, where $\Gamma=\Gamma_0(4\cdot 49)\cap\Gamma_1(7)$ if $m\not\equiv 0 \ (\text{mod} 7)$ or $\Gamma=\Gamma_0(4\cdot 49)$ if $m\equiv0 \ (\text{mod} 7)$. We introduce several modular forms in $\mathcal{M}_2(\Gamma)$ for preparation.

Let \[
D(\tau):=\sum_{n=1}^{\infty}\sigma(n)q^n,
\] where $\sigma(n)=\sum_{d\mid n, d>0}d$.
\begin{proposition}
    $D(\tau)\mid S_{m,M}$ is a cusp form of weight 2 in $\Gamma_0(M^2)$ if $m\neq 0$.
\end{proposition}
\begin{proof}
    Set\[G_2(\tau)=\sum_{c\in\mathbb{Z}}\sum_{d\in\mathbb{Z}-\{0\}}\frac{1}{(c\tau+d)^2}=2\zeta(2)-8\pi^2\sum_{n=1}^{\infty}\sigma(n)q^n.\] The function $G_2(\tau)-\pi/\text{Im}(\tau)$ is weight-2 invariant under $\text{SL}_2(\mathbb{Z})$\cite{Diamond2008AFC}. Set\[
    G_{2,N}(\tau) =G_2(\tau)-NG_2(N\tau),
    \] Then $G_{2,N}(\tau)=(NG_2(N\tau)-\pi/\text{Im}(\tau))-(G_2(\tau)-N\pi/\text{Im}(N\tau))$ satisfies weight-2 invariant under $\Gamma_0(N)$. Therefore $G_{2,M}(\tau)\mid S_{m,M}\in \mathcal{M}_2(\Gamma_0(M^2))$, so does $D$.
\end{proof}

The cusp form space $S_{2}^{\mathrm{new}}(\Gamma_0(49))$ has one dimension. We denote the newform in $S_{2}^{\mathrm{new}}(\Gamma_0(49))$ to be
        \[G:= q +  q^{2} -  q^{4} - 3 q^{8} - 3 q^{9} + O(q^{10}).\]
This newform, labeled as 49.2.a.a, has complex multiplication on $\mathbb{Q}(\sqrt{-7})$, is associated with elliptic curve $y^2+xy=x^3-x^2-2x-1$. This $G$ will appear in the representation of $(\mathcal{H}_{m,7}+\frac{1}{2}\Lambda_{1,m,7})|U_4|S_{a,7}$ and $(\mathcal{H}_{m,7}+\frac{1}{2}\Lambda_{1,m,7})|U_4\otimes\chi_{-7}^2$ later.

In the finite dimensional linear space $\mathcal{M}_2(\Gamma)$, two modular forms are the same if their coefficients of the first $B$ terms agree with each other, where $B$ is an integer that only depend on $k$ and $\Gamma$. More precisely, we introduce the Sturm bound\cite{sturmbound1}\cite{sturmbound2}:
\begin{proposition}
    For any space $\mathcal{M}_k(N,\chi)$ of modular forms of weight $k$, level $N$, and character $\chi$, the Sturm bound is the integer\[
B(\mathcal{M}_k(N,\chi)) := \left\lfloor \frac{km}{12}\right\rfloor,\] where $
m:=[\text{SL}_2(\mathbb{Z}):\Gamma]$.\\
Suppose $f=\sum_{n\ge 0}a_n q^n\in \mathcal{M}_k(N,\chi)$ and $g=\sum_{n\ge 0}b_n q^n\in \mathcal{M}_k(N,\chi)$. If $a_n=b_n$ for all $n\le B(M_k(N,\chi))$ then $f=g$.
\end{proposition}

\begin{lemma}
    \[\text{SL}_2(\mathbb{Z}):\Gamma_0(N_1)\cap\Gamma_1(N_2)=N_1\prod_{p\mid N_1}(1+\frac{1}{p})\phi(N_2),\] where $\phi(N_1)=\#(\mathbb{Z}/N\mathbb{Z})$.
\end{lemma}

In the case $\Gamma=\Gamma_0(4\cdot 49)\cap \Gamma_1(7)$, $B(\mathcal{M}_2(\Gamma))=336$, and in the case $\Gamma=\Gamma_0(4\cdot 49)$, $B(\mathcal{M}_2(\Gamma))=56$.Now we start to describe the modular form version of our main theorem. Note that the $m=0$ part is already computed in \cite{BRINGMANN_2018} and \cite{Mertens_2016}.

\begin{theorem}\label{35} Let $\chi_{-7}$ be the non-principal character modulo 7.\\
    \rm (1)\it We have the following equation in $\mathcal{M}_2(\Gamma_0(4\cdot 49))$:
    \[(\mathcal{H}\theta_{0,7})\mid U_4\otimes\chi_{-7}^2+2D_1^{(7,1)}\mid S_{7,6}+2D_1^{(7,2)}\mid S_{7,3}+2D_1^{(7,3)}\mid S_{7,5}\]\[=\frac{1}{4}D\otimes\chi_{-7}^2+\frac{1}{24}D\otimes \chi_{-7}\otimes(\chi_7-1)+\frac{1}{4}G\otimes\chi_{-7}^2.\]
    \rm (2)\it We have the following equation in $\mathcal{M}_2(\Gamma_0(4\cdot 49)\cap\Gamma_1(7))$:\\
    \[\begin{cases}
        (\mathcal{H}\theta_{1,7})\mid U_4\mid S_{7,1}+D_1^{(7,3)}\mid S_{7,1}+D_1^{(7,5)}\mid S_{7,1}=\frac{1}{3}D\mid S_{7,1},\\
        (\mathcal{H}\theta_{1,7})\mid U_4\mid S_{7,2}+\frac{1}{2}D_1^{(7,3)}\mid S_{7,2}+\frac{1}{2}D_1^{(7,4)}\mid S_{7,2} \\
        \quad \quad =\frac{7}{24}D\mid S_{7,2}+\frac{1}{8}G\mid S_{7,2},\\
        (\mathcal{H}\theta_{1,7})\mid U_4\mid S_{7,3}=\frac{1}{4}D\mid S_{7,3},\\
        (\mathcal{H}\theta_{1,7})\mid U_4\mid S_{7,4}=\frac{1}{4}D\mid S_{7,4}-\frac{1}{4}G\mid S_{7,4},\\
        (\mathcal{H}\theta_{1,7})\mid U_4\mid S_{7,5}+D_1^{(7,6)}\mid S_{7,5}+D_1^{(7,2)}\mid S_{7,5}=\frac{1}{3}D\mid S_{7,5},\\
        (\mathcal{H}\theta_{1,7})\mid U_4\mid S_{7,6}=\frac{1}{4}D\mid S_{7,6}.
    \end{cases}\]
    \[\begin{cases}
        (\mathcal{H}\theta_{2,7})\mid U_4\mid S_{7,1}+\frac{1}{2}D_1^{(7,1)}\mid S_{7,1}+\frac{1}{2}D_1^{(7,6)}\mid S_{7,1} \\
        \quad \quad =\frac{7}{24}D\mid S_{7,1}+\frac{1}{8}G\mid S_{7,1},\\
        (\mathcal{H}\theta_{2,7})\mid U_4\mid S_{7,2}=\frac{1}{4}D\mid S_{7,2}-\frac{1}{4}G\mid S_{7,2},\\
        (\mathcal{H}\theta_{2,7})\mid U_4\mid S_{7,3}=\frac{1}{4}D\mid S_{7,3},\\
        (\mathcal{H}\theta_{2,7})\mid U_4\mid S_{7,4}+D_1^{(7,3)}\mid S_{7,4}+D_1^{(7,6)}\mid S_{7,4}=\frac{1}{3}D\mid S_{7,4},\\
        (\mathcal{H}\theta_{2,7})\mid U_4\mid S_{7,5}=\frac{1}{4}D\mid S_{7,5},\\
        (\mathcal{H}\theta_{2,7})\mid U_4\mid S_{7,6}+D_1^{(7,4)}\mid S_{7,6}+D_1^{(7,5)}\mid S_{7,6}=\frac{1}{3}D\mid S_{7,6}.
    \end{cases}\]
    \[\begin{cases}
        (\mathcal{H}\theta_{3,7})\mid U_4\mid S_{7,1}=\frac{1}{4}D\mid S_{7,1}-\frac{1}{4}G\mid S_{7,1},\\
        (\mathcal{H}\theta_{3,7})\mid U_4\mid S_{7,2}+D_1^{(7,1)}\mid S_{7,2}+D_1^{(7,2)}\mid S_{7,2}=\frac{1}{3}D\mid S_{7,2},\\
        (\mathcal{H}\theta_{3,7})\mid U_4\mid S_{7,3}+D_1^{(7,4)}\mid S_{7,3}+D_1^{(7,6)}\mid S_{7,3}=\frac{1}{3}D\mid S_{7,3},\\
        (\mathcal{H}\theta_{3,7})\mid U_4\mid S_{7,4}+\frac{1}{2}D_1^{(7,2)}\mid S_{7,4}+\frac{1}{2}D_1^{(7,5)}\mid S_{7,4} \\
        \quad \quad =\frac{7}{24}D\mid S_{7,4}+\frac{1}{8}G\mid S_{7,1},\\
        (\mathcal{H}\theta_{3,7})\mid U_4\mid S_{7,5}=\frac{1}{4}D\mid S_{7,5},\\
        (\mathcal{H}\theta_{3,7})\mid U_4\mid S_{7,6}=\frac{1}{4}D\mid S_{7,6}.
    \end{cases}\]
\end{theorem} 
\begin{proof}
    By proposition \ref{31}, \[(\mathcal{H}\theta_{0,7})\mid U_4\otimes\chi_{-7}^2+2D_1^{(7,1)}\mid S_{7,6}+2D_1^{(7,2)}\mid S_{7,3}+2D_1^{(7,3)}\mid S_{7,5}\in \mathcal{M}_2(\Gamma_0(4\cdot 49)),\]
    and the following are inside $\mathcal{M}_2(\Gamma_0(4\cdot 49)\cap\Gamma_1(7))$:
    \[(\mathcal{H}\theta_{1,7})\mid U_4\otimes\chi_{-7}^2+\frac{1}{2}D_1^{(7,3)}\mid S_{7,2}+\frac{1}{2}D_1^{(7,4)}\mid S_{7,2}+D_1^{(7,3)}\mid S_{7,1}+D_1^{(7,6)}\mid S_{7,5}+D_1^{(7,2)}\mid S_{7,5}\]\[+D_1^{(7,5)}\mid S_{7,1},\]
    \[(\mathcal{H}\theta_{2,7})\mid U_4\otimes\chi_{-7}^2+\frac{1}{2}D_1^{(7,1)}\mid S_{7,1}+\frac{1}{2}D_1^{(7,6)}\mid S_{7,1}+D_1^{(7,4)}\mid S_{7,6}+D_1^{(7,3)}\mid S_{7,4}+D_1^{(7,6)}\mid S_{7,4}\]\[+D_1^{(7,5)}\mid S_{7,6},\]
    \[(\mathcal{H}\theta_{3,7})\mid U_4\otimes\chi_{-7}^2+\frac{1}{2}D_1^{(7,2)}\mid S_{7,4}+\frac{1}{2}D_1^{(7,5)}\mid S_{7,4}+D_1^{(7,1)}\mid S_{7,2}+D_1^{(7,4)}\mid S_{7,3}+D_1^{(7,6)}\mid S_{7,3}\]\[+D_1^{(7,2)}\mid S_{7,2}.\]
    By pervious discussion, if $m\neq 0$, $D\mid S_{m,7}$ and $G$ are modular forms of weight 2 in $\Gamma_0(49)$, therefore they are in $\mathcal{M}_2(\Gamma_0(4\cdot 49)\cap\Gamma_1(7))$ and $\mathcal{M}_2(\Gamma_0(4\cdot 49))$. For every equation in our theorem, we compare the first 337 coefficients of $q$-series if $m\neq 0$ and the first 57 coefficients of $q$-series if $m=0$. They all coincides, therefore these equations are true.
\end{proof}

This theorem offers a method to compute $H_{m,7}(p)$. However, we need to compute the coefficients of CM newform $G$ first. In the next section, we reveal the connection between the Fourier expansion of $G$ and $\chi_{-7}(x)x$ where $p=x^2+7y^2$. 

\section{THE COEFFICIENTS OF NEWFORM G}
In \cite{zindulka2024sumshurwitzclassnumbers}, the author construct modular forms generated by $\chi(x)x$ where $p=x^2+ny^2$ for some given $n$ and $\chi$. Using this modular form, the author successfully linked $H_{m,M}(n)$ with the previous $x$ when $m=6,8$. Thanks to Zindulka's idea, we may construct a modular form of weight 2 under $\mathcal{M}_2(\Gamma_0(4\cdot 49))$.

\begin{proposition}
    Let \[\Psi_k(\chi,\tau)=\frac{1}{2}\sum_{n=1}^{\infty}(\sum_{\substack{x,y\in\mathbb{Z}\\[1pt]x^2+ky^2=n}}\chi(x)x)q^n,\] where $\chi$ is an odd dirichlet character and $k\in\mathbb{Z}_{\geq2}$, then \[\Psi_7(\chi_{-7},\tau)\in\mathcal{M}_2(\Gamma_0(196)).\]
\end{proposition}
\begin{proof}
    If we set $\theta_0(\tau):=\sum_{y\in\mathbb{Z}}q^{y^2}$ and $\theta(\chi,1,\tau):=\frac{1}{2}\sum_{x\in\mathbb{Z}}\chi(x)xq^{x^2}$. Then \[\Psi_k(\chi,n)=\theta(\chi,1,\tau)\cdot(\theta_0\mid V(k)(\tau)\]. From \cite{zindulka2024sumshurwitzclassnumbers}, $\theta(\chi,1,\tau)\in \mathcal{S}_{\frac{3}{2}}(\Gamma_0(4N_{\chi}^2),\chi\chi_{-4})$ and $(\theta_0\mid V(k))(\tau)\in \mathcal{M}_{\frac{1}{2}}(\Gamma_0(4k),\left(\frac{k}{\cdot}\right))$.

    In our situation, $\theta(\chi_{-7},1,\tau)\in\mathcal{S}_{\frac{3}{2}}(\Gamma_0(196),\chi_{-7}\chi_{-4})$, $(\theta_0\mid V(7))(\tau)\in \mathcal{M}_{\frac{1}{2}}(\Gamma_0(28),\left(\frac{7}{\cdot}\right))$. By checking $n$ in range $\{0,1,2,...,27\}$, $\chi_{-4}\chi_{-7}\left(\frac{7}{\cdot}\right)=1$. Therefore \[\Psi_7(\chi_{-7},\tau)=\theta(\chi_{-7},1,\tau)\cdot(\theta_0\mid V(7)(\tau)\in \mathcal{M}_2(\Gamma_0(196)).\]
\end{proof}

Next, we are going to represent $\Psi_7(\chi_{-7},\tau)$ as linear combinations of $G$. 

\begin{lemma}\label{42}
    \[\Psi_7(\chi_{-7},\tau)=G-G\mid U(2)+4G\mid U(4).\]
\end{lemma}
\begin{proof}
    $G\in\mathcal{S}_2(\Gamma_0(49))$, therefore $G$, $G\mid U(2)$ and $G\mid U(4)$ are all in $\mathcal{S}_2(\Gamma_0(196))$. By Sturm bound, we need to compute the first 56 coefficients of $q$-series of both sides of the equation. The first 56 coefficients of $q$-series of both sides of the equation coincides, therefore $\Psi_7(\chi_{-7},\tau)=G-G\mid U(2)+4G\mid U(4)$.
\end{proof}

Lemma \ref{42} build a bridge between the Fourier expansion of $G$ and $\Psi_7(\chi_{-7},\tau)$. For an odd prime $p\neq 7$, the $q^p$ term vanish in $G\mid U(2)$, which means the coefficient of $q^p$ term of $G_7$ equals to \[\frac{1}{2}\sum_{\substack{x,y\in\mathbb{Z}\\[1pt]x^2+7y^2=p}}\chi_{-7}(x)x.\]

\begin{lemma}\label{43}
    For odd prime $p\neq 7$ if $\left(\frac{7}{p}\right)=1$ (which means $p\equiv 1,2,4\mod 7$), then there exist unique pair $(x,y)\in\mathbb{Z}_{> 0}\times\mathbb{Z}_{> 0}$ such that $p=x^2+7y^2$.
\end{lemma}
\begin{proof}
See (2.17) of \cite{primenxy}.
\end{proof}

Lemma \ref{43} map every odd prime $p$ that $p\equiv 1,2,4\mod 7$ to a unique $x$. Comparing the coefficients of equations in Thm \ref{35} and notice that $H_{m,M}(n)=H_{-m,M}(n)$, one instantly get the following formulas:

\begin{theorem}\label{main2}
    Let $p>2$ be a odd prime and $p\equiv 1,2,4 \mod 7$. Then There exist exactly one positive integer $x$ such that $p=x^2+7y^2$. Let $\chi_{-7}$ denotes the non principle character modulo 7. We have\[
    H_{a,7}(p)=\begin{cases}
\frac{p+1}{4}+\frac{1}{2}\chi_{-7}(x)x & \text{if } a=0 \\
\frac{p+1}{4}-\frac{1}{2}\chi_{-7}(x)x & \text{if } a=\pm 1, p\equiv 4\mod 7, \text{or }a=\pm 2, p\equiv 1\mod 7, \\ & \text{or }a=\pm 3, p\equiv 1\mod 7,\\
\frac{7p+7}{24}+\frac{1}{4}\chi_{-7}(x)x & \text{if } a=\pm 1, p\equiv 2\mod 7, \text{or }a=\pm 3, p\equiv 4\mod\\
\frac{7p-17}{24}+\frac{1}{4}\chi_{-7}(x)x & \text{if } a=\pm 2, p\equiv 1\mod 7.
\end{cases}
    \]
\end{theorem}

At the end, we use this consequence to fill up the diagram of $H_{m,7}(p)$.

\vspace{7em}

\begin{table}[!ht]
    \centering
    \begin{tabular}{|c|c|c|c|c|}
    \hline
        ~ & $m=0$ & $m=\pm 1$ & $m=\pm 2$ & $m=\pm 3$ \\ \hline
        $p\equiv 1\mod 7$ & $\frac{p+1}{4}+\frac{1}{2}\chi_{-7}(x)x$ & $\frac{p+1}{3}$ & $\frac{7p-17}{24}+\frac{1}{4}\chi_{-7}(x)x$ & $\frac{p+1}{4}-\frac{1}{2}\chi_{-7}(x)$ \\ \hline
        $p\equiv 2\mod 7$ & $\frac{p+1}{4}+\frac{1}{2}\chi_{-7}(x)x$ & $\frac{7p+7}{24}+\frac{1}{4}\chi_{-7}(x)x$ & $\frac{p+1}{4}-\frac{1}{2}\chi_{-7}(x)$ & $\frac{p-2}{3}$ \\ \hline
        $p\equiv 3\mod 7$ & $\frac{p+1}{3}$ & $\frac{p+1}{4}$ & $\frac{p+1}{4}$ & $\frac{p-2}{3}$ \\ \hline
        $p\equiv 4\mod 7$ & $\frac{p+1}{4}+\frac{1}{2}\chi_{-7}(x)x$ & $\frac{p+1}{4}-\frac{1}{2}\chi_{-7}(x)$ & $\frac{p-2}{3}$ & $\frac{7p+7}{24}+\frac{1}{4}\chi_{-7}(x)x$ \\ \hline
        $p\equiv 5\mod 7$ & $\frac{p+1}{3}$ & $\frac{p-2}{3}$ & $\frac{p+1}{4}$ & $\frac{p+1}{4}$ \\ \hline
        $p\equiv 6\mod 7$ & $\frac{p-5}{3}$ & $\frac{p+1}{4}$ & $\frac{p+1}{3}$ & $\frac{p+1}{4}$ \\ \hline
    \end{tabular}
\end{table}

\bibliographystyle{unsrt}
\bibliography{refs}

$\,$

$\,$

\end{document}